\UseRawInputEncoding
\documentclass[12pt]{amsart}
\usepackage{setspace, amsmath, amsthm, amssymb, amsfonts, amscd, epic, graphicx, ulem, dsfont}
\usepackage[T1]{fontenc}
\usepackage{multirow}
\usepackage{bbm}
\usepackage{enumerate}

\makeatletter \@namedef{subjclassname@2010}{
  \textup{2020} Mathematics Subject Classification}
\makeatother

\newtheorem{thm}{Theorem}[section]

\newtheorem{cor}[thm]{Corollary}
\newtheorem{lem}[thm]{Lemma}
\newtheorem{pro}[thm]{Proposition}

\theoremstyle{remark}
\newtheorem*{rema}{Remark}

\theoremstyle{definition}

\newtheorem{exa}[thm]{\textbf{Example}}

\newcommand{\ran}{\text{\rm{ran}}}

\newcommand{\R}{\mathbb{R}}

\newcommand{\N}{\mathbb{N}}

\newcommand{\C}{\mathbb{C}}

\begin{document}

\title[Fuglede-Putnam theorem]{Unbounded generalizations of the Fuglede-Putnam theorem and applications to the commutativity of self-adjoint operators}
\author[S. Dehimi, M. H. Mortad and A. Bachir]{Souheyb Dehimi, Mohammed Hichem Mortad$^*$ and Ahmed Bachir}

\thanks{* Corresponding author.}
\date{}


\address{(The first author) Department of Mathematics, Faculty of Mathematics and Informatics,
University of Mohamed El Bachir El Ibrahimi, Bordj Bou Arréridj,
El-Anasser 34030, Algeria.}

\email{souheyb.dehimi@univ-bba.dz, sohayb20091@gmail.com}

\address{(The corresponding author) Department of
Mathematics, University of Oran 1, Ahmed Ben Bella, B.P. 1524, El
Menouar, Oran 31000, Algeria.}

\email{mhmortad@gmail.com, mortad.hichem@univ-oran1.dz.}

\address{(The third author) Department of Mathematics, College of Science, King Khalid University, Abha, Saudi
Arabia.}

\email{abishr@kku.edu.sa, bachir\_ahmed@hotmail.com}

\begin{abstract}In this article, we prove and disprove several generalizations of unbounded versions of the Fuglede-Putnam
theorem. As applications, we give conditions guaranteeing the
commutativity of a bounded self-adjoint operator with an unbounded
closed symmetric operator.
\end{abstract}

\maketitle

\section{Essential background}

All operators considered here are linear but not necessarily
bounded. If an operator is bounded and everywhere defined, then it
belongs to $B(H)$ which is the algebra of all bounded linear
operators on $H$ (see \cite{Mortad-Oper-TH-BOOK-WSPC} for its
fundamental properties).

Most unbounded operators that we encounter are defined on a subspace
(called domain) of a Hilbert space. If the domain is dense, then we
say that the operator is densely defined. In such case, the adjoint
exists and is unique.

Let us recall a few basic definitions about non-necessarily bounded
operators. If $S$ and $T$ are two linear operators with domains
$D(S)$ and $D(T)$ respectively, then $T$ is said to be an extension
of $S$, written as $S\subset T$, if $D(S)\subset D(T)$ and $S$ and
$T$ coincide on $D(S)$.

An operator $T$ is called closed if its graph is closed in $H\oplus
H$. It is called closable if it has a closed extension. The smallest
closed extension of it is called its closure and it is denoted by
$\overline{T}$ (a standard result states that a densely defined $T$
is closable iff $T^*$ has a dense domain, and in which case
$\overline{T}=T^{**}$). If $T$ is closable, then
\[S\subset T\Rightarrow \overline{S}\subset
\overline{T}.\] If $T$ is densely defined, we say that $T$ is
self-adjoint when $T=T^*$; symmetric if $T\subset T^*$; normal if
$T$ is \textit{closed} and $TT^*=T^*T$.

The product $ST$ and the sum $S+T$ of two operators $S$ and $T$ are
defined in the usual fashion on the natural domains:

\[D(ST)=\{x\in D(T):~Tx\in D(S)\}\]
and
\[D(S+T)=D(S)\cap D(T).\]

In the event that $S$, $T$ and $ST$ are densely defined, then
\[T^*S^*\subset (ST)^*,\]
with the equality occurring when $S\in B(H)$. If $S+T$ is densely
defined, then
\[S^*+T^*\subset (S+T)^*\]
with the equality occurring when $S\in B(H)$.

Let $T$ be a linear operator (possibly unbounded) with domain $D(T)$
and let $B\in B(H)$. Say that $B$ commutes with $T$ if
\[BT\subset TB.\]
In other words, this means that $D(T)\subset D(TB)$ and
\[BTx=TBx,~\forall x\in D(T).\]

Let $A$ be an injective operator (not necessarily bounded) from
$D(A)$ into $H$. Then $A^{-1}: \ran(A)\rightarrow H$ is called the
inverse of $A$, with $D(A^{-1})=\ran(A)$.

If the inverse of an unbounded operator is bounded and everywhere
defined (e.g. if $A:D(A)\to H$ is closed and bijective), then $A$ is
said to be boundedly invertible. In other words, such is the case if
there is a $B\in B(H)$ such that
\[AB=I\text{ and } BA\subset I.\]
If $A$ is boundedly invertible, then it is closed.

The resolvent set of $A$, denoted by $\rho(A)$, is defined by
\[\rho(A)=\{\lambda\in\C:~\lambda I-A\text{ is bijective and }(\lambda I-A)^{-1}\in B(H)\}.\]

The complement of $\rho(A)$, denoted by $\sigma(A)$,
\[\sigma(A)=\C\setminus \rho(A)\]
is called the spectrum of $A$.

Recall also that the product of two closed operators need not be
closed (see \cite{Mortad-cex-BOOK}). However, and it is known (among
other results), that $TS$ is closed if $T$ is closed and $S\in B(H)$
or if $T^{-1}$ is in $B(H)$ and $S$ is closed.

If a symmetric operator $T$ is such that $\langle Tx,x\rangle\geq0$
for all $x\in D(T)$, then we say that $T$ is positive, and we write
$T\geq0$. When $T$ is self-adjoint and $T\geq0$, then we can define
its unique positive self-adjoint square root, which we denote by
$\sqrt T$.

If $T$ is densely defined and closed, then $T^*T$ (and $TT^*$) is
self-adjoint and positive (a celebrated result due to von-Neumann,
see e.g. \cite{SCHMUDG-book-2012}). So, when $T$ is closed then
$T^*T$ is self-adjoint and positive whereby it is legitimate to
define its square root. The unique positive self-adjoint square root
of $T^*T$ is denoted by $|T|$. It is customary to call it the
absolute value or modulus of $T$. If $T$ is closed, then (see e.g.
Lemma 7.1 in \cite{SCHMUDG-book-2012})
\[D(T)=D(|T|)\text{ and } \|Tx\|=\||T|x\|,~\forall x\in D(T).\]

Next, we recall some definitions of unbounded non-normal operators.
A densely defined operator $A$ with domain $D(A)$ is called
hyponormal if
\[D(A)\subset D(A^*)\text{ and } \|A^*x\|\leq\|Ax\|,~\forall x\in D(A).\]

A densely defined linear operator $A$ with domain $D(A)\subset H$,
is said to be subnormal when there are a Hilbert space $K$ with
$H\subset K$, and a normal operator $N$ with $D(N)\subset K$ such
that
\[D(A)\subset D(N)\text{ and } Ax=Nx \text{ for all } x\in D(A).\]

In the end, we recall some basic facts about matrices of
non-necessarily bounded operators. Let $H$ and $K$ be two Hilbert
spaces and let $A:H\oplus K\to H\oplus K$ (we may also use $H\times
K$ instead of $H\oplus K$) be defined by
\[
A=\left(
      \begin{array}{cc}
        A_{11} & A_{12} \\
        A_{21} & A_{22} \\
      \end{array}
    \right)
\]
where $A_{11}\in L(H)$, $A_{12}\in L(K,H)$, $A_{21}\in L(H,K)$ and
$A_{22}\in L(K)$ are not necessarily bounded operators. If $A_{ij}$
has a domain $D(A_{ij})$ with $i,j=1,2$, then
\[D(A)=(D(A_{11})\cap D(A_{21}))\times (D(A_{12})\cap D(A_{22}))\]
is the natural domain of $A$. So if $(x_1,x_2)\in D(A)$, then
\[A\left(
     \begin{array}{c}
       x_1 \\
       x_2 \\
     \end{array}
   \right)=\left(
             \begin{array}{c}
               A_{11}x_1+A_{12}x_2\\
               A_{21}x_1+A_{22}x_2 \\
             \end{array}
           \right).
\]

Also, recall that the adjoint of $\left(
      \begin{array}{cc}
        A_{11} & A_{12} \\
        A_{21} & A_{22} \\
      \end{array}
    \right)$ is not always $\left(
      \begin{array}{cc}
        A^*_{11} & A^*_{21} \\
        A^*_{12} & A^*_{22} \\
      \end{array}
    \right)$ (even when all domains are dense including the main domain $D(A)$) as
    known counterexamples show. Nonetheless, e.g.
    \[\left(
        \begin{array}{cc}
          A & 0 \\
          0 & B \\
        \end{array}
      \right)^*=\left(
        \begin{array}{cc}
          A^* & 0 \\
          0 & B^* \\
        \end{array}
      \right)\text{ and } \left(
                            \begin{array}{cc}
                              0 & C \\
                              D & 0 \\
                            \end{array}
                          \right)^*=\left(
                            \begin{array}{cc}
                              0 & D^* \\
                              C^* & 0 \\
                            \end{array}
                          \right)
    \]
if $A$, $B$, $C$ and $D$ are all densely defined.

\section{Introduction}

The aim of this paper is twofold. In the first part, we obtain some
generalizations of the Fuglede-Putnam theorem involving unbounded
operators. In the second part, we apply the Fuglede-Putnam theorem
to obtain conditions guaranteeing the commutativity of self-adjoint
operators, one of them is bounded.

Recall that the original version of the Fuglede-Putnam theorem
reads:

\begin{thm}\label{Fug-Put UNBD A BD}(\cite{FUG}, \cite{PUT}) If $A\in B(H)$ and if $M$
and $N$ are normal (non necessarily bounded) operators, then
\[AN\subset MA\Longrightarrow AN^*\subset M^*A.\]
\end{thm}

There have been many generalizations of the Fuglede-Putnam theorem
since Fuglede's paper. However, most generalizations were devoted to
relaxing the normality assumption. Apparently, the first
generalization of the Fuglede theorem to an unbounded $A$ was
established in \cite{Nussbaum-1969}. Then the first generalization
involving unbounded operators of the Fuglede-\textit{Putnam} theorem
is:

\begin{thm}\label{Fug-Put-MORTAD-PAMS-2003} If $A$ is a closed and symmetric operator and if $N$ is an unbounded normal operator, then
\[AN\subset N^*A\Longrightarrow AN^*\subset NA\]
whenever $D(N)\subset D(A)$.
\end{thm}

In fact, the previous result was established in \cite{MHM1} under
the assumption of the self-adjointness of $A$. However, and by
scrutinizing the proof in \cite{MHM1} or
\cite{Mortad-Thesis-Edinburgh-2003}, it is seen that only the
closedness and the symmetricity of $A$ were needed. Other unbounded
generalizations may be consulted in
\cite{Mortad-Fuglede-Putnham-All-unbd} and
\cite{Bens-MORTAD-FUGLEDE-DEHIMI} and some of the references
therein. In the end, readers may wish to consult the survey
\cite{Mortad-FUG-PUT SURVEY BOOK} exclusively devoted to the
Fuglede-Putnam theorem and its applications.

In the second part of this manuscript, we continue the
investigations initiated in the thesis
\cite{Mortad-Thesis-Edinburgh-2003}, and then in
\cite{Gustafson-Mortad-II} inter alia. More precisely, we show that
if $B\in B(H)$ is self-adjoint and $A$ is densely defined, closed
and symmetric, then $BA\subset AB$ given that $AB$ or $BA$ is e.g.
normal.

\section{Generalizations of the Fuglede-Putnam theorem}

If a densely defined operator $N$ is normal, then so is its adjoint.
However, if $N^*$ is normal, then $N^{**}$ does not have to be
normal (unless $N$ itself is closed). A simple counterexample is to
take the identity operator $I_D$ restricted to some unclosed dense
domain $D\subset H$. Then $I_D$ cannot be normal for it is not
closed. But, $(I_D)^*=I$ which is the full identity on the entire
$H$, is obviously normal. Notice in the end that if $N$ is a densely
defined closable operator, then $N^*$ is normal if and only if
$\overline{N}$ is.

The first improvement is that in the very first version by B.
Fuglede, the normality of the operator is not needed as only the
normality of its closure will do. This observation has already
appeared in \cite{Boucif-Dehimi-Mortad}, but we reproduce the proof
here.

\begin{thm}\label{fuglede-type closure normal THM}
Let $B\in B(H)$ and let $A$ be a densely defined and closable
operator such that $\overline{A}$ is normal. If $BA\subset AB$, then
\[BA^*\subset A^*B.\]
\end{thm}

\begin{proof}
Since $\overline{A}$ is normal, $\overline{A}^*=A^*$ remains normal.
Now,
\begin{align*}
BA\subset AB\Longrightarrow& B^*A^*\subset A^*B^* \text{ (by taking adjoints)}\\
\Longrightarrow &B^*\overline{A}\subset \overline{A}B^* \text{ (by using the classical Fuglede theorem)}\\
\Longrightarrow& BA^*\subset A^*B \text{ (by taking adjoints
again),}
\end{align*}
establishing the result.
\end{proof}

\begin{rema}
Notice that $BA^*\subset A^*B$ does not yield $BA\subset AB$ even in
the event of the normality of $A^*$ (see \cite{Mortad-cex-BOOK}).
\end{rema}

Let us now turn to the extension of the Fuglede-Putnam version. A
similar argument to the above one could be applied.

\begin{thm}\label{fuglede-Putnam type closure normal THM}
Let $B\in B(H)$ and let $N,M$ be densely defined closable operators
such that $\overline{N}$ and $\overline{M}$ are normal. If
$BN\subset MB$, then
\[BN^*\subset M^*B.\]
\end{thm}

\begin{proof} Since $BN\subset MB$, it ensues that $B^*M^*\subset
N^*B^*$. Taking adjoints again gives $B\overline{N}\subset
\overline{M}B$. Now, apply the Fuglede-Putnam theorem to the normal
$\overline{N}$ and $\overline{M}$ to get the desired conclusion
\[BN^*\subset M^*B.\]
\end{proof}

Jab{\l}o\'{n}ski et al. obtained in \cite{Jablonski et al 2014} the
following version.

\begin{thm}\label{jablonsky et al FUG type THM}
If $N$ is a normal (bounded) operator and if $A$ is a closed densely
defined operator with $\sigma(A)\neq \C$, then:
\[NA\subset AN\Longrightarrow g(N)A\subset Ag(N)\]
for any bounded complex Borel function $g$ on $\sigma(N)$. In
particular, we have $N^*A\subset AN^*$.
\end{thm}

\begin{rema}
It is worth noticing that B. Fuglede obtained, long ago, in
\cite{FUG-1954-cexp-proper extension} a unitary $U\in B(H)$ and a
closed and symmetric $T$ with domain $D(T)\subset H$ such that
$UT\subset TU$ but $U^*T\not\subset TU^*$.
\end{rema}

Next, we give a generalization of Theorem \ref{jablonsky et al FUG
type THM} to an unbounded $N$, and as above, only the normality of
$\overline{N}$ is needed.

\begin{thm}\label{THM JAblonsky's et al generalization UNB}
Let $p$ be a one variable complex polynomial. If $N$ is a densely
defined closable operator such that $\overline{N}$ is normal and if
$A$ is a densely defined operator with $\sigma[p(A)]\neq \C$, then
\[NA\subset AN\Longrightarrow N^*A\subset AN^*\]
whenever $D(A)\subset D(N)$.
\end{thm}

\begin{rema}
This is indeed a generalization of the bounded version of the
Fuglede theorem. Observe that when $A,N\in B(H)$, then
$\overline{N}=N$, $D(A)=D(N)=H$, and $\sigma[p(A)]$ is a compact
set.
\end{rema}

\begin{proof}First, we claim that $\sigma(A)\neq \C$, whereby $A$ is
closed. Let $\lambda$ be in $\C\setminus \sigma[p(A)]$. Then, and as
in \cite{Dehimi-Mortad-squares-polynomials}, we obtain
\[p(A)-\lambda I=(A-\mu_1 I)(A-\mu_2I)\cdots (A-\mu_nI)\]
for some complex numbers $\mu_1$, $\mu_2$, $\cdots$, $\mu_n$. By
consulting again \cite{Dehimi-Mortad-squares-polynomials}, readers
see that $\sigma(A)\neq \C$.

Now, let $\lambda\in \rho(A)$. Then
\[NA\subset AN\Longrightarrow NA-\lambda N\subset AN-\lambda N=(A-\lambda I)N.\]
Since $D(A)\subset D(N)$, it is seen that $NA-\lambda N=N(A-\lambda
I)$. So
\[ N(A-\lambda I)\subset (A-\lambda I)N\Longrightarrow (A-\lambda
I)^{-1}N\subset N(A-\lambda I)^{-1}.\]

Since $\overline{N}$ is normal, we may now apply Theorem
\ref{fuglede-type closure normal THM} to get
\[(A-\lambda I)^{-1}N^*\subset N^*(A-\lambda I)^{-1}\]
because $(A-\lambda I)^{-1}\in B(H)$. Hence
\[N^*A-\lambda N^*\subset N^*(A-\lambda I)\subset (A-\lambda I)N^*=AN^*-\lambda N^*.\]
But
\[D(AN^*)\subset D(N^*)\text{ and }D(N^*A)\subset D(A)\subset D(N)\subset D(\overline{N})=D(N^*).\]
Thus, $D(N^*A)\subset D(AN^*)$, and so
\[N^*A\subset AN^*,\]
as needed.
\end{proof}

Now, we present a few consequences of the preceding result. The
first one is given without proof.

\begin{cor}
If $N$ is a densely defined closable operator such that
$\overline{N}$ is normal and if $A$ is an unbounded self-adjoint
operator with $D(A)\subset D(N)$, then
\[NA\subset AN\Longrightarrow N^*A\subset AN^*.\]
\end{cor}

\begin{cor}
If $N$ is a densely defined closable operator such that
$\overline{N}$ is normal and if $A$ is a boundedly invertible
operator, then
\[NA\subset AN\Longrightarrow N^*A\subset AN^*.\]
\end{cor}

\begin{proof}We may write
\[NA\subset AN\Longrightarrow NAA^{-1}\subset ANA^{-1}\Longrightarrow A^{-1}N\subset NA^{-1}.\]
Since $A^{-1}\in B(H)$ and $\overline{N}$ is normal, Theorem
\ref{fuglede-type closure normal THM} gives
\[A^{-1}N^*\subset N^*A^{-1} \text{ and so }N^*A\subset AN^*,\]
as needed.

\end{proof}

A Putnam's version seems impossible to obtain unless strong
conditions are imposed. However, the following special case of a
possible Putnam's version is worth stating and proving. Besides, it
is somewhat linked to the important notion of anti-commutativity
(cf. \cite{Vasilescu anticommuting}).

\begin{pro}
If $N$ is a densely defined closable operator such that
$\overline{N}$ is normal and if $A$ is a densely defined operator
with $\sigma(A)\neq \C$, then
\[NA\subset -AN\Longrightarrow N^*A\subset -AN^*\]
whenever $D(A)\subset D(N)$.
\end{pro}

\begin{proof}Consider
\[\widetilde{N}=\left(
                  \begin{array}{cc}
                    N & 0 \\
                    0 & -N \\
                  \end{array}
                \right)\text{ and }\widetilde{A}=\left(
                                                   \begin{array}{cc}
                                                     0 & A \\
                                                     A & 0 \\
                                                   \end{array}
                                                 \right)
\]
where $D(\widetilde{N})=D(N)\oplus D(N)$ and
$D(\widetilde{A})=D(A)\oplus D(A)$. Then $\overline{\widetilde{N}}$
is normal and $\widetilde{A}$ is closed. Besides
$\sigma(\widetilde{A})\neq \C$. Now
\[\widetilde{N}\widetilde{A}=\left(
                                                   \begin{array}{cc}
                                                     0 & NA \\
                                                     -NA & 0 \\
                                                   \end{array}
                                                 \right)\subset \left(
                                                   \begin{array}{cc}
                                                     0 & -AN \\
                                                     AN & 0 \\
                                                   \end{array}
                                                 \right)=\widetilde{A}\widetilde{N}\]
for $NA\subset -AN$. Since $D(\widetilde{A})\subset
D(\widetilde{N})$, Theorem \ref{THM JAblonsky's et al generalization
UNB} applies, i.e. it gives $\widetilde{N}^*\widetilde{A}\subset
\widetilde{A}\widetilde{N}^*$ which, upon examining their
           entries, yields the required result.
\end{proof}

We finish this section by giving counterexamples to some
"generalizations".

\begin{exa}(\cite{Mortad-Fuglede-Putnham-All-unbd}) Consider the unbounded linear operators $A$ and $N$ which are defined by
\[Af(x)=(1+|x|)f(x)\text{ and } Nf(x)=-i(1+|x|)f'(x)\]
(with $i^2=-1$) on the domains
\[D(A)=\{f\in L^2(\R): (1+|x|)f\in L^2(\R)\}\]
and
\[D(N)=\{f\in L^2(\R): (1+|x|)f'\in L^2(\R)\}\]
respectively, and where the derivative is taken in the
distributional sense. Then $A$ is a boundedly invertible, positive,
self-adjoint unbounded operator. As for $N$, it is an unbounded
normal operator $N$ (details may consulted in
\cite{Mortad-Fuglede-Putnham-All-unbd}). It was shown that such that
\[AN^*=NA\text{ but }AN\not\subset N^*A \text{ and }N^*A\not\subset AN\]
(in fact $ANf\neq N^*Af$ for all $f\neq 0$).

So, what this example is telling us is that $NA=AN^*$ (and not just
an "inclusion"), that $N$ and $N^*$ are both normal, $\sigma(A)\neq
\C$ (as $A$ is self-adjoint), but $NA\not\subset AN^*$.
\end{exa}

This example can further be beefed up to refute certain possible
generalizations.

\begin{exa}\label{hghgdfsdzretrutioyoyoypypypy}(Cf. \cite{Mortad-Fuglede-ultimate generalization}) There exist a closed operator $T$ and a normal $M$ such that $TM\subset MT$ but $TM^*\not\subset M^*T$ and $M^*T\not\subset
TM^*$. Indeed, consider
\[M=\left(
      \begin{array}{cc}
        N^* & 0 \\
        0 & N \\
      \end{array}
    \right)\text{ and } T=\left(
                            \begin{array}{cc}
                              0 & 0 \\
                              A & 0 \\
                            \end{array}
                          \right)
\]
where $N$ is normal with domain $D(N)$ and $A$ is closed with domain
$D(A)$ and such that $AN^*=NA$\text{ but }$AN\not\subset N^*A$
\text{ and }$N^*A\not\subset AN$ (as defined above). Clearly, $M$ is
normal and $T$ is closed. Observe that $D(M)=D(N^*)\oplus D(N)$ and
$D(T)=D(A)\oplus L^2(\R)$. Now,
\[TM=\left(
                            \begin{array}{cc}
                              0 & 0 \\
                              A & 0 \\
                            \end{array}
                          \right)\left(
      \begin{array}{cc}
        N^* & 0 \\
        0 & N \\
      \end{array}
    \right)=\left(
              \begin{array}{cc}
                0_{D(N^*)} & 0_{D(N)} \\
                AN^* & 0 \\
              \end{array}
            \right)=\left(
              \begin{array}{cc}
                0 & 0_{D(N)} \\
                AN^* & 0 \\
              \end{array}
            \right)
    \]
where e.g. $0_{D(N)}$ is the zero operator restricted to $D(N)$.
Likewise
\[MT=\left(
      \begin{array}{cc}
        N^* & 0 \\
        0 & N \\
      \end{array}
    \right)\left(
                            \begin{array}{cc}
                              0 & 0 \\
                              A & 0 \\
                            \end{array}
                          \right)=\left(
                            \begin{array}{cc}
                              0 & 0 \\
                              NA & 0 \\
                            \end{array}
                          \right).\]
Since $D(TM)=D(AN^*)\oplus D(N)\subset D(NA)\oplus L^2(\R)=D(MT)$,
it ensues that $TM\subset MT$. Now, it is seen that
\[TM^*=\left(
                            \begin{array}{cc}
                              0 & 0 \\
                              A & 0 \\
                            \end{array}
                          \right)\left(
      \begin{array}{cc}
        N & 0 \\
        0 & N^* \\
      \end{array}
    \right)=\left(
              \begin{array}{cc}
                0 & 0_{D(N^*)} \\
                AN & 0 \\
              \end{array}
            \right)\]
and
\[M^*T=\left(
      \begin{array}{cc}
        N & 0 \\
        0 & N^* \\
      \end{array}
    \right)\left(
                            \begin{array}{cc}
                              0 & 0 \\
                              A & 0 \\
                            \end{array}
                          \right)=\left(
                            \begin{array}{cc}
                              0 & 0 \\
                              N^*A & 0 \\
                            \end{array}
                          \right).\]

Since $ANf\neq N^*Af$ for any $f\neq 0$, we infer that
$TM^*\not\subset M^*T$ and $M^*T\not\subset TM^*$.
\end{exa}

\section{Some applications to the commutativity of self-adjoint operators}

In \cite{Alb-Spain}, \cite{Dehimi-Mortad-INVERT},
\cite{Gustafson-Mortad-I}, \cite{Gustafson-Mortad-II},
\cite{Jung-Mortad-Stochel}, \cite{MHM1},
\cite{Mortad-Thesis-Edinburgh-2003}, \cite{MHM7},
\cite{Mortad-FUG-PUT SURVEY BOOK}, and \cite{Reh}, the question of
the self-adjointness of the normal product of two self-adjoint
operators was tackled in different settings (cf.
\cite{Benali-Mortad}). In all cases, the commutativity of the
operators was reached.

Here, we deal with the similar question where the unbounded
(operator) factor is closed and symmetric which, and it is known,
differs from self-adjointness (the two classes can behave quite
differently, cf. \cite{Mortad-cex-BOOK}).

First, we give a perhaps known lemma (cf. Lemmata 2.1 \& 2.2 in
\cite{Gustafson-MZ-positive-PROD-1968}). See also
\cite{Meziane-Mortad-I} for the case of normality.

\begin{lem}\label{product}(\cite{Jung-Mortad-Stochel})
Let $A$ and $B$ be self-adjoint operators. Assume that $B\in B(H)$
and $BA \subseteq AB$. Then the following assertions hold:
\begin{enumerate}
\item[(i)]  $AB$ is a self-adjoint operator and $AB=\overline{BA}$,
\item[(ii)] if $A$ and $B$ are positive so is $AB$.
\end{enumerate}
\end{lem}

We shall also have need for the following result:

\begin{lem}\label{PAMS 03 LEMMA squr rt BA subset AB }
Let $B\in B(H)$ be self-adjoint. If $BA\subset AB$ where $A$ is
closed, then $f(B)A\subset Af(B)$ for any continuous function $f$ on
$\sigma(B)$. In particular, and if $B$ is positive, then $\sqrt
BA\subset A\sqrt B$.
\end{lem}

\begin{rema}
In fact, the previous lemma was shown in (\cite{MHM1}, Proposition
1) under the assumption "$A$ being unbounded and self-adjoint", but
by looking closely at its proof, we see that only the closedness of
$A$ was needed (cf. \cite{Bernau JAusMS-1968-square root} and
\cite{Jablonski et al 2014}).
\end{rema}

We are now ready to state and prove the first result of this
section.

\begin{thm}\label{BDM-THM-AB self-adjoint}
Let $A$ be an unbounded closed and symmetric operator with domain
$D(A)$, and let $B\in B(H)$ be positive. If $AB$ is normal, then
$BA\subset AB$, and so $AB$ is self-adjoint. Also, $\overline{BA}$
is self-adjoint.

Besides, $B|A|\subset |A|B$, and so $|A|B$ is self-adjoint and
positive. Moreover, $|A|B=\overline{B|A|}$.
\end{thm}

\begin{proof}Since $B\in B(H)$ is self-adjoint, we have
$(BA)^*=A^*B$ and $BA^*\subset (AB)^*$. Now, write
\[B(AB)=BAB\subset BA^*B\subset (AB)^*B.\]
Since $AB$ and $(AB)^*$ are both normal, the Fuglede-Putnam theorem
applies and gives
\[B(AB)^*\subset (AB)^{**}B=\overline{AB}B=AB^2,\]
i.e.
\[B^2A\subset B^2A^*\subset B(AB)^*\subset AB^2.\]
Since $A$ is closed and $B\in B(H)$ is positive, Lemma \ref{PAMS 03
LEMMA squr rt BA subset AB } gives
\[BA\subset AB.\]
To show that $AB$ is self-adjoint, we proceed as follows: Observe
that
\[BA\subset BA^*\subset (AB)^*.\]
Since we also have $BA\subset AB$, we now know that
\[BAx=ABx=(AB)^*x\]
for all $x\in D(A)$. This says that $AB$ and $(AB)^*$ coincide on
$D(A)$. Denoting the restrictions of the latter operators to $D(A)$
by $T$ and $S$ respectively, it is seen that
\[T-S\subset 0, ~T\subset AB,\text{ and }S\subset (AB)^*.\]
Hence
\[(AB)^*-AB\subset T^*-S^*\subset (T-S)^*=0.\]
Since $D(AB)=D[(AB)^*]$ thanks to the normality of $AB$, it ensues
that $AB=(AB)^*$, that is, $AB$ is self-adjoint.

Now, we show that $\overline{BA}$ is self-adjoint. First, we show
that $\overline{BA}$ is normal. Clearly $BA^*\subset A^*B$ for we
already know that $BA\subset AB$. Hence
\[BA^*A\subset A^*BA\subset A^*AB.\]
Therefore
\[\overline{BA}(\overline{BA})^*=ABA^*B\subset AA^*B^2\]
and
\[(\overline{BA})^*\overline{BA}=A^*BAB\subset A^*AB^2.\]
By Lemma \ref{product}, it is seen that both of $AA^*B^2$ and
$AA^*B^2$ are self-adjoint. By the maximality of self-adjoint
operators, it ensues that
\[\overline{BA}(\overline{BA})^*=AA^*B^2\text{ and }(\overline{BA})^*\overline{BA}=A^*AB^2.\]
Since $AB$ is self-adjoint, $(AB)^2$ is self-adjoint. But
\[(AB)^2=ABAB\subset AA^*B^2\]
and so $(AB)^2=AA^*B^2$. Similarly,
\[ABAB\subset A^*BAB\subset A^*AB^2\]
or $(AB)^2=A^*AB^2$. Therefore, we have shown that
\[(\overline{BA})^*\overline{BA}=\overline{BA}(\overline{BA})^*.\]
In other words, $\overline{BA}$ is normal.

To infer that $\overline{BA}$ is self-adjoint, observe that
$BA\subset AB$ gives $\overline{BA}\subset AB$, but because normal
operators are maximally normal, we obtain $\overline{BA}=AB$, from
which we derive the self-adjointness of $\overline{BA}$.

To show the last claim of the theorem, consider again $BA^*A\subset
A^*AB$. So, $B|A|\subset |A|B$ by the spectral theorem say. Since
$B\geq0$, Lemma \ref{product} gives the self-adjointness and the
positivity of $|A|B$, as well as $|A|B=\overline{B|A|}$. This
completes the proof.
\end{proof}

\begin{rema}
Under the assumptions of the preceding theorem (by consulting
\cite{Boucif-Dehimi-Mortad}), we have:
\[|AB|=|\overline{BA}|=|A|B=\overline{B|A|}.\]
\end{rema}

\begin{cor}\label{23/02/2022 ...COR}
Let $A$ be an unbounded closed and symmetric operator and let $B\in
B(H)$ be positive. Suppose that $AB$ is normal. Then
\[BA \text{ is closed} \Longrightarrow A\text{ is self-adjoint}.\]
In particular, if $B$ is invertible, then $A$ is self-adjoint.
\end{cor}

\begin{proof}By Theorem \ref{BDM-THM-AB self-adjoint}, $\overline{BA}$ is
self-adjoint and $\overline{BA}=AB$. Hence
\[BA^*\subset (AB)^*=(\overline{BA})^*=\overline{BA}.\]
So, when $BA$ is closed, $BA^*\subset BA$. Therefore, $D(A^*)\subset
D(A)$, and so $D(A)=D(A^*)$. Thus, $A$ is self-adjoint, as required.
\end{proof}

\begin{cor}
Let $A$ be an unbounded closed and symmetric operator with domain
$D(A)$, and let $B\in B(H)$ be positive. If $BA^*$ is normal, then
$BA\subset AB$, and so $BA^*$ is self-adjoint.
\end{cor}

\begin{proof}Since $BA^*$ is normal, so is $(BA^*)^*=AB$. To obtain
the desired conclusion, one just need to apply Theorem
\ref{BDM-THM-AB self-adjoint}.
\end{proof}

The case of the normality of $BA$ was unexpectedly trickier. After a
few attempts, we have been able to show the result.

\begin{thm}\label{BDM-THM-BA self-adjoint}
Let $A$ be an unbounded closed and symmetric operator with domain
$D(A)$, and let $B\in B(H)$ be self-adjoint. Assume $BA$ is normal.
Then $A$ is necessarily self-adjoint.

If we further assume that $B$ is positive, then $BA$ becomes
self-adjoint and $BA=AB$.
\end{thm}

We are now ready to show Theorem \ref{BDM-THM-BA self-adjoint}.

\begin{proof}First, recall that since $BA$ is normal, $BA$ is closed
and $D(BA)=D[(BA)^*]$.

Write
\[A(BA)\subset A^*BA=(BA)^*A.\]
Since $BA$ is normal and $D(BA)=D(A)$, Theorem
\ref{Fug-Put-MORTAD-PAMS-2003} is applicable and it gives
\[A(BA)^*\subset (BA)^{**}A=\overline{BA}A=BA^2,\]
i.e. $AA^*B\subset BA^2$. Since $A$ is symmetric, we may push the
previous inclusion to further obtain $AA^*B\subset BAA^*$, that is,
$|A^*|^2B\subset B|A^*|^2$.

Next, we claim that $B|A^*|$ is closed too. To see that, observe
that as $B\in B(H)$, then $(BA)^*=A^*B$. Hence
$\overline{BA}=(A^*B)^*$ or $BA=(A^*B)^*$ because $BA$ is already
closed. By Lemma 11 in \cite{CG}, the last equation is equivalent to
$(|A^*|B)^*=B|A^*|$ which gives the closedness of $B|A^*|$ as
needed.

Now, we have
\[B|A^*|(B|A^*|)^*=B|A^*|^2B\subset B^2|A^*|^2.\]
It then follows by Corollary 1 in \cite{DevNussbaum-von-Neumann}
that
\[B|A^*|(B|A^*|)^*=B^2|A^*|^2\]
for $B|A^*|(B|A^*|)^*$, $B^2$, and $|A^*|^2$ are all self-adjoint.
The self-adjointness of $B|A^*|(B|A^*|)^*$ also implies that
$B^2|A^*|^2$ is self-adjoint as well, i.e.
\[B^2|A^*|^2=(B^2|A^*|^2)^*=|A^*|^2B^2.\]
In particular, $B^2|A^*|^2$ is closed. So, Proposition 3.7 in
\cite{Dehimi-Mortad-INVERT} implies that $B|A^*|^2$ is closed.

The next step is to show that $B|A^*|^2$ is normal. As
$|A^*|^2B\subset B|A^*|^2$, it ensues that
\[B|A^*|^2(B|A^*|^2)^*=B|A^*|^4B\subset B^2|A^*|^4\]
and
\[(B|A^*|^2)^*B|A^*|^2=|A^*|^2B^2|A^*|^2\subset B^2|A^*|^4.\]
Since $B|A^*|^2(B|A^*|^2)^*$, $(B|A^*|^2)^*B|A^*|^2$, $B^2$, and
$|A^*|^2$ are all self-adjoint, Corollary 1 in
\cite{DevNussbaum-von-Neumann} yields
\[B|A^*|^2(B|A^*|^2)^*=(B|A^*|^2)^*B|A^*|^2~~(=B^2|A^*|^4).\]
Therefore, $B|A^*|^2$ is normal. So, since $B\in B(H)$ is
self-adjoint and $|A^*|^2$ is self-adjoint and positive, it follows
by Theorem 1.1 in \cite{Gustafson-Mortad-II} that $B|A^*|^2$ is
self-adjoint and $B|A^*|^2=|A^*|^2B$.

By applying Theorem 10 in \cite{Bernau JAusMS-1968-square root}, it
is seen that
\[B|A^*|=|A^*|B\]
due to the self-adjointness and the positivity of $|A^*|$.

We now have all the necessary tools to establish the
self-adjointness of $A$. Indeed,
\begin{align*}
D(A^*)=D(|A^*|)=D(B|A^*|)&=D(|A^*|B)\\&=D(A^*B)=D[(BA)^*]=D(BA)=D(A).
\end{align*}
Thus, $A$ is self-adjoint as it is already symmetric.

Finally, when $B\in B(H)$ is positive and since $A$ is self-adjoint,
$(BA)^*=AB$ is normal. By Theorem \ref{BDM-THM-AB self-adjoint},
$AB$ is self-adjoint or $(BA)^*$ is self-adjoint. In other words,
\[BA=(BA)^*=AB,\]
and this marks the end of the proof.
\end{proof}

Generalizations to weaker classes than normality vary. Notice in
passing that in \cite{Dehimi-Mortad-INVERT}, the self-adjointness of
$BA$ was established for a positive $B\in B(H)$ and an unbounded
self-adjoint $A$ such that $BA$ is hyponormal and
$\sigma(BA)\neq\C$. The next result is of the same kind.

\begin{pro}Let $B\in B(H)$ be positive and let $A$ be a
densely defined closed symmetric operator. If $(AB)^*$ is subnormal
or if $BA^*$ is closed and subnormal, then $BA\subset AB$.

Moreover, if $A$ is self-adjoint, then $AB$ is self-adjoint.
Besides, $AB=\overline{BA}$.
\end{pro}

\begin{proof}The proof relies on a version of the Fuglede-Putnam theorem
obtained by J. Stochel in \cite{STO-asymm-Fuglede-Putnam-PAMS-2001}.
Write
\[B[(AB)^*]^*=B(AB)^{**}=BAB\subset BA^*B\subset (AB)^*B.\]
Since $(AB)^*$ is subnormal, Theorem 4.2 in
\cite{STO-asymm-Fuglede-Putnam-PAMS-2001} yields
\[B^2A\subset B^2A^*\subset B(AB)^*\subset (AB)^{**}B=AB^2.\]

The same inclusion is obtained in the event of the subnormality of
$BA^*$. Indeed, write
\[B(BA^*)^*=BAB\subset BA^*B.\]
Applying again Theorem 4.2 in
\cite{STO-asymm-Fuglede-Putnam-PAMS-2001} gives
\[B(BA^*)\subset (BA^*)^*B=AB^2.\]
Therefore, and as above, we obtain $B^2A\subset AB^2$.

Now, since $B\geq 0$ and $A$ is closed, it follows that $BA\subset
AB$.

Finally, when $A$ is self-adjoint, Lemma \ref{product} implies that
$AB$ is self-adjoint and $AB=\overline{BA}$, as needed.
\end{proof}

There are still more cases to investigate. As is known, if $N\in
B(H)$ is such that $N^2$ is normal, then $N$ need not be normal (cf.
\cite{Mortad-square-root-normal}). The same applies for the class of
self-adjoint operators.

The first attempted generalization is the following: Let $A,B$ be
two self-adjoint operators, where $B$ is positive, and such that
$(AB)^n$ is normal for some $n\in\N$ such that $n\geq2$. Does it
follow that $AB$ is self-adjoint?

The answer is negative even when $A$ and $B$ are $2\times 2$
matrices. This is seen next:

\begin{exa}Take
\[A=\left(
      \begin{array}{cc}
        0 & 1 \\
        1 & 0 \\
      \end{array}
    \right)
 \text{ and }B=\left(
                 \begin{array}{cc}
                   0 & 0 \\
                   0 & 1 \\
                 \end{array}
               \right).
 \]
Then $A$ is self-adjoint and $B$ is positive (it is even an
orthogonal projection). Also,
\[AB=\left(
      \begin{array}{cc}
        0 & 1 \\
        0 & 0 \\
      \end{array}
    \right)\text{ whilst }(AB)^n=\left(
                 \begin{array}{cc}
                   0 & 0 \\
                   0 & 0 \\
                 \end{array}
               \right)\]
    for all $n\geq2$. In other words, $AB$ is not self-adjoint while all
    $(AB)^n$, $n\geq2$, are patently self-adjoint.
\end{exa}

Let us pass to other possible generalizations.

\begin{pro}
Let $B\in B(H)$ be positive and let $A$ be a closed and symmetric
operator. Assume $AB^n$ is normal for a certain positive integer
$n\in\N$. Then
\begin{enumerate}
  \item $BA\subset AB$ (hence $BA$ is symmetric).
  \item If it is further assumed that $B$ is invertible, then $A$ is self-adjoint. Besides, all of
  $AB^{1/n}$ and $B^{1/n}A$ are self-adjoint for all $n\geq 1$.
\end{enumerate}
\end{pro}

\begin{proof}\hfill
\begin{enumerate}
  \item Since $B^n$ is positive for all $n$ and $AB^n$ is normal, it follows by Theorem \ref{BDM-THM-AB self-adjoint} that $AB^n$ is
  self-adjoint and $B^nA\subset AB^n$. By Lemma \ref{PAMS 03 LEMMA squr rt BA subset AB
  }, it is seen that $BA\subset AB$.
  \item Since $AB^n$ is normal and $B^nA$ is closed (as $B^n$ is
  invertible), Corollary \ref{23/02/2022 ...COR} yields the
  self-adjointness of $A$.

  Finally, since $BA\subset AB$ and $B$ is positive, it follows that
  $B^{1/n}A\subset AB^{1/n}$, from which we derive the
  self-adjointness of $AB^{1/n}$ and $\overline{B^{1/n}A}=B^{1/n}A$,
  as suggested.
\end{enumerate}
\end{proof}

Similarly, we have:

\begin{pro}\label{23/01/2022}
Let $B\in B(H)$ be positive and let $A$ be a closed and symmetric
operator. Assume that $B^nA$ is normal for some positive integer
$n\in\N$. Then $A$ and $BA$ are self-adjoint, and $BA=AB$.
\end{pro}

One of the tools to prove this result is:

\begin{lem}\label{BnA closed implies BA closed LEMMA}(Cf. Proposition 3.7
in \cite{Dehimi-Mortad-INVERT}) Let $B\in B(H)$ and let $A$ be an
arbitrary operator such that $B^nA$ is closed for some integer
$n\geq2$. Suppose further that $BA$ is closable. Then $BA$ is
closed.
\end{lem}

\begin{proof}Let $(x_p)$ be in $D(B^nA)$ and such that $x_p\to x$ and $BAx_p\to
y$. Since $B^{n-1}\in B(H)$, $B^nAx_p\to B^{n-1}y$. Since $B^nA$ is
closed, we obtain $x\in D(B^nA)=D(A)$. Since $BA\subset
\overline{BA}$ and $x\in D(BA)$, we have
\[BAx=\overline{BA}x=\lim_{p\to \infty}BAx_p=y\]
by the definition of the closure of an operator. We have therefore
shown that $BA$ is closed, as wished.
\end{proof}

Now, we show Proposition \ref{23/01/2022}.

\begin{proof} Since $B^n$ is positive, Theorem \ref{BDM-THM-BA
self-adjoint} gives both the self-adjointness of $A$ and $B^nA$.
Moreover, $B^nA=AB^n$. Using Lemma \ref{PAMS 03 LEMMA squr rt BA
subset AB } or else, we get $BA\subset AB$ (only the inclusion
suffices to finish the proof). The equation $B^nA=AB^n$ contains the
closedness of $B^nA$ which, by a glance at Lemma \ref{BnA closed
implies BA closed LEMMA}, yields $BA=AB$ by consulting Lemma
\ref{product}.
\end{proof}

\end{document}